\documentclass[12pt]{amsart}
\usepackage{latexsym}

\usepackage{amssymb}

\newcommand{\ac}{\circlearrowright}

\theoremstyle{plain}

\newtheorem{coro}{Corollary}
\newtheorem{prop}{Proposition}
\newtheorem{lem}{Lemma}

\newtheorem{dfn}{Definition}
\newtheorem{nott}{Notation}

\theoremstyle{remark}

\begin{document}

\title{On the entropy for group actions on the circle}
\author{Eduardo Jorquera}

\maketitle

\vspace{-1cm}

\begin{abstract}
We show that for a finitely generated group of $C^2$ circle diffeomorphisms,
the entropy of the action equals the entropy of the restriction of the
action to the non-wandering set.
\end{abstract}

\section{Introduction}

Let $(X,dist)$ be a compact metric space and $G$ a group of homeomorphisms of $X$
generated by a finite family of elements $\Gamma \!=\! \{g_1, \ldots ,g_n \}$.
To simplify, we will always assume that $\Gamma$ is symmetric, that is,
$g^{-1} \!\in\! \Gamma$ for every $g \! \in \! \Gamma$. For each $n \in \mathbb{N}$
we denote by $B_{\Gamma}(n)$ the ball of radius $n$ in $G$ (w.r.t. $\Gamma$),
that is, the set of elements $f \!\in\! G$ which may be written in the form
$f \!=\! g_{i_m} \cdots g_{i_1}$ for some $m \!\leq\! n$ and
$g_{i_j} \!\in\! \Gamma$. For $g \in G$ we let $\|f \| = \| f \|_{\Gamma} 
= min\{n \!: f \in B_{\Gamma}(n) \}$

As in the classical case, given $\varepsilon > 0$ and $n \!\in\! \mathbb{N}$, two
points $x,y$ in $X$ are said to be $(n,\varepsilon)$-separated if there exists
$g \!\in\! B_{\Gamma}(n)$ such that $dist (g(x),g(y)) \geq \varepsilon$. A subset
$A \subset X$ is $(n,\varepsilon)$-separated if all $x \neq y$ in $A$ are
$(n,\varepsilon)$-separated. We denote by $s(n,\varepsilon)$
the maximal possible cardinality (perhaps infinite)
of a $(n,\varepsilon)$-separated set. 
The topological entropy for the action at the scale $\varepsilon$ is defined by
$$h_{\Gamma}(G \ac X,\varepsilon) = \limsup_{n \uparrow \infty}
\frac{\log \big( s(n,\varepsilon) \big)}{n},$$
and the topological entropy is defined by
$$h_{\Gamma}(G \ac X) =
\lim_{\varepsilon\downarrow 0} h_{\Gamma}(G \ac X,\varepsilon).$$
Notice that, although $h_{\Gamma} (G \ac X, \varepsilon)$ depends on the
system of generators, the properties of having zero, positive, or infinite
entropy, are independent of this choice.

The definition above was proposed in \cite{GLW} as an extention of the classical
topological entropy of single maps (the definition extends to pseudo-groups of
homeomorphisms, and hence is suitable for applications in Foliation Theory).
Indeed, for a homeomorphism $f$, the topological entropy of the action of
\hspace{0.1cm} $\mathbb{Z} \!\sim\! \langle f \rangle$ \hspace{0.1cm} equals
two times the (classical) topological entropy of $f$. Nevertheless, the
functorial properties of this notion remain unclear. For example,
the following fundamental question is open.

\vspace{0.5cm}

\noindent{\bf General Question.} Is it true that $h_{\Gamma}(G \ac X)$
is equal to $h_{\Gamma}(G \ac \Omega)$~?\\

\noindent Here $\Omega \, =\, \Omega \,(G \ac X)$ denotes the {\em non-wandering} set of
the action, or in other words

$
\Omega \ = \{
\begin{array}{lc}
x \in X \!: \mbox{ for every neighborhood } U \mbox{ of } x,
\mbox{ we have } & \\
 f(U) \cap U \neq \emptyset, \mbox{  for some } f\neq id \mbox{ in } G. &
\end{array} \hspace{-0.4cm} \}
$

\noindent This is a closed invariant set whose complement $\Omega^c$
corresponds to the {\em wandering set} of the action.

The notion of topological entropy for group actions is quite appropriate in
the case where $X$ is a one--dimensional manifold. In fact, in this case, the
topological entropy is necessarily finite ({\em cf.} \S \ref{background}).
Moreover, in the case of actions by diffeomorphisms,
the dichotomy $h_{top} = 0$ or $h_{top} > 0$ is well understood.
Indeed, according to a result
originally proved by Ghys, Langevin, and Walczak, for groups of $C^2$
diffeomorphisms \cite{GLW}, and extended by Hurder to groups of $C^1$
diffeomorphisms (see for instance \cite{wal}), we have
$h_{top} > 0$ if and only if there exists a resilient orbit for the
action. This means that there exist a group element $f$ contracting
by one side to a fixed point $x_0$, and another element $g$ which sends
$x_0$ into its basin of contraction by $f$.

The results of this work give a positive answer to the General Question above
in the context of group actions on one--dimensional manifolds under certain
mild assumptions.

\vspace{0.5cm}

\noindent{\bf Theorem A.} {\em If $G$ is a finitely generated subgroup
of $\mathrm{Diff}_{+}^{2}(\mathrm{S}^{1})$, then for every finite
system of generators $\Gamma$ of $G$, we have}
$$h_{\Gamma}\,(G \ac S^1) = \, h_{\Gamma}\,(G \ac \Omega) \,.$$

\vspace{0.35cm}

Our proof for Theorem A actually works in the Denjoy class $C^{1 + bv}$, and applies to general
codimension-one foliations on compact manifolds. In the class $C^{1 + Lip}$, it is quite possible that
we could give an alternative proof using standard techniques from Level Theory \cite{CC,hector}.

It is unclear whether Theorem A extends to actions of lower regularity. However, it still holds
under certain algebraic hypotheses. In fact, (quite unexpectedly) 
the regularity hypothesis is used to rule out the existence of elements $f \!\in\! G$
that fix some connected component of the wandering set and which are {\em distorted}: that is,
those elements which satisfy
$$\lim_{n \to \infty} \frac{\| f^n \|}{n} = 0.$$
Actually, for the equality between the entropies it suffices
to require that no elememt in $G$ be {\em sub-exponentially distorted}.
In other words, it suffices to require that, for each element
$f \!\in\! G$ with infinite order, there exist a non--decreasing function
$q \!: \mathbb{N} \rightarrow \mathbb{N}$ (depending on $f$) with
sub--exponential growth satisfying  \,$q (\| f^n \|) \geq n , \,$
for every $n \!\in\! \mathbb{N}$. This is an algebraic
condition which is satisfied by many groups, as for example nilpotent or
free groups. (We refer the reader to \cite{cal} for a nice discussion on distorted
elements.) Under this hypothesis, the following result holds.

\vspace{0.6cm}

\noindent{\bf Theorem B.} {\em If $G$ is a finitely generated subgroup
of $\mathrm{Homeo}_{+}(\mathrm{S}^{1})$ without sub--exponentially distorted elements,
then for every finite system of generators $\Gamma$ of $G$, we have}
$$h_{\Gamma}(G \ac S^1) = h_{\Gamma}\,(G \ac \Omega) \,.$$

\vspace{0.4cm}

The entropy of general group actions and distorted elements seem to be related in an
interesting manner. Indeed, though the topological entropy of a single homeomorphism $f$ may
be equal to zero, if this map appears as a sub--exponentially distorted element inside
an acting group, then this map may create positive entropy for the group action.


\section{Some background}
\label{background}

In this work we will consider the normalized length on the circle, and every homeomorphism
will be orientation preserving.

We begin by noticing that if $G$ is a finitely generated group of
circle homeomorphisms and $\Gamma$ is a finite generating system for $G$, then
for all $n \in \mathbb{N}$ and all $\varepsilon > 0$ one has 
\begin{equation}
s(n,\varepsilon) \leq \frac{1}{\varepsilon} \# B_{\Gamma}(n).
\label{no}
\end{equation}
Indeed, let $A$ be a $(n,\varepsilon)$-separated set of cardinality 
$s(n,\varepsilon)$. Then for every two adjacent points $x,y$ in $A$ there 
exists $f \in B_{\Gamma}(n)$ such that $dist(f(x),f(y)) \geq \varepsilon$. 
For a fixed $f$, the intervals $[f(x),f(y)]$ which appear have disjoint 
interior. Since the total length of the circle is 1, any given $f$ can 
be used in this construction at most $1/\varepsilon$ times, which 
immediately gives (\ref{no}).

Notice that, taking the logarithm at both sides of (\ref{no}), 
dividing by $n$, and passing to the limits, this gives
$$h_{\Gamma}(G \ac S^1) \leq gr_{\Gamma}(G),$$
where $gr_{\Gamma}(G)$ denotes the {\em growth} of $G$ with respect to $\Gamma$,
that is,
$$gr_{\Gamma}(G) = \lim_{n \rightarrow \infty} \frac{\log(\# B_{\Gamma}(n))}{n}.$$
Some easy consequences of this fact are the following ones:\\

\vspace{0.01cm}

\noindent -- If $G$ has sub-exponential growth, that is, if $gr_{\Gamma}(G) \!=\! 0$ (in particular,
if $G$ is nilpotent, or if $G$ is the Grigorchuk-Maki's group considered in \cite{growth}),
then $h_{\Gamma}(G \ac S^1) = 0$ for all finite generating systems $\Gamma$.\\

\vspace{0.01cm}

\noindent -- In the general case, if \,$\# \Gamma = q \geq 1$,\, then from the relations
$$\# B_{\Gamma}(n) \leq 1+ \sum_{j=1}^{n} 2q (2q-1)^{j-1} =
\left \{ \begin{array} {l}
1 + \big( \frac{q}{q-1} \big) \big( (2q-1)^{n}-1 \big), \hspace{0.32cm} \hfill q \geq 2,\\
1 + 2n, \hfill q = 1, \end{array} \right.$$
one concludes that
$$h_{\Gamma}(G \ac S^1) \leq \log (2q-1).$$
This shows in particular that the entropy of the action of $G$ on $S^{1}$ is finite.
Notice that this may be also deduced from the probabilistic arguments of \cite{DKN}
(see Th\'eor\`eme D therein). However, these arguments only yield the weaker estimate
$h_{\Gamma}(G \ac S^1) \leq \log (2q)$ when $\Gamma$ has cardinality $q$.

\vspace{0.2cm}


\section{Some preparation for the proofs}

The statement of our results are obvious when the non-wandering set of the action equals the whole
circle. Hence, we will assume in what follows that $\Omega$ is a proper subset of $\mathrm{S}^1$,
and we will currently denote by $I$ some of the connected components of the complement
of $\Omega$. Let $Est(I)$ denote the stabilizer of $I$ in $G$.

\vspace{0.2cm}

\begin{lem} {\em The stabilizer $Est(I) \,$ is either trivial or infinite cyclic.}
\end{lem}

\begin{proof} The (restriction to $I$ of the) nontrivial elements of
$\, Est(I)|_{I}  \,$ have no fixed points, for otherwise these
points would be non-wandering.
Thus $\, Est(I)|_{I} \,$ acts freely on $I$, and according to H\"older
Theorem \cite{ghys,navas},
its action is semiconjugate to an action by translations. We
claim that, if $\, Est(I)|_{I} \,$ is nontrivial, then it is infinite
cyclic. Indeed, if not then the corresponding group of translations is
dense. This implies that the preimage by the semiconjugacy of any point
whose preimage is a single point corresponds to a non-wandering point
for the action. Nevertheless, this contradicts the fact that $I$ is
contained in $\Omega^c$.

If $Est(I)|_{I}$ is trivial then $f|_I$ is trivial for every $f \in Est(I)$,
and hence $f$ itself must be the identity. We then conclude that
$\, Est(I) \,$ is trivial.

Analogously, $\, Est(I) \,$ is cyclic if $\, Est(I)|_{I} \,$ is
cyclic. In this case, $\, Est(I)|_{I} \,$ is generated by the
restriction to the interval
$I$ of the generator of $Est(I)$. \end{proof}

\vspace{0.1cm}

\begin{dfn} A connected component $I$ of $\Omega^c$ will be
called of \textit{type 1} if $\, Est(I) \,$ is trivial, and
will be called of {type 2} if $\, Est(I) \,$ is infinite cyclic.
\end{dfn}

\vspace{0.15cm}

Notice that the families of connected components of type 1 and 2 are invariant, that is, for
each $f \!\in\! G$ the interval $f(I)$ is of type 1 (resp. of type 2) if $I$ is of type 1
(resp. of type 2). Moreover, given two connected components of type 1 of $\Omega^c$,
there exists at most one element in $G$ sending the former into the latter. Indeed,
if $f(I) \!=\! g(I)$ then $g^{-1} f$ is in the stabilizer of $I$, and hence
$f \!=\! g$ if $I$ is of type 1.

\vspace{0.2cm}

\begin{lem}
Let $x_1,\ldots,x_m$ be points contained in a single type 1 connected
component of $\Omega^c$. If for some $\varepsilon \!>\! 0$ the points
$x_i,x_j$ are $(\varepsilon,n)$-separated for every $i \neq j$,
then $m \leq 1 + \frac{1}{\varepsilon}$.
\label{lema0}
\end{lem}

\begin{proof}
Let $I \! = ]a,b[$ be the connected component of type 1 of $\Omega^c$
containing the points $x_1,\ldots,x_m$. After renumbering the $x_i$'s,
we may assume that $a < x_1 < x_2 < \ldots < x_m < b$. For each
$1 \leq i \leq m-1$ one can choose an element
$g_i \in B_{\Gamma}(n)$ such that
$dist (g_i(x_i),g_i(x_{i+1})) \geq \varepsilon$. Now, since $I$
is of type 1, the intervals $]g_i (x_i), g_i (x_{i+1})[$ are two
by two disjoint. Therefore, the number of these intervals times
the minimal length among them is less than or equal to 1. This
gives $(m-1) \varepsilon \leq 1$, thus proving the lemma.
\end{proof}

\vspace{0.35cm}

The case of connected components $I$ of type 2 of $\Omega^c$ is much more complicated than
the one of type 1 connected components. The difficulty is related to the fact that, if the
generator of the stabilizer of $I$ is sub-exponentially distorted in $G$, then this would imply
the existence of exponentially many $(n,\varepsilon)$-separated points inside $I$, and hence
a relevant part of the entropy would be ``concentrated'' in $I$. To deal with this problem,
for each connected component $I$ of type 2 of $\Omega^c$ we denote by $p_I$ its middle point,
and then we define $\ell_I \!\!: G \rightarrow \mathbb{N}_0$ as follows. Let $h$ be the
generator of the stabilizer of $I$ such that $h(x)>x$ for all $x$ in $I$. For each
$f \!\in\! G$ the element $fhf^{-1}$ is the generator of the stabilizer of $f(I)$ with
the analogous property. We then let $\ell_I (f)= \lvert r \lvert$, where $r$ is the
unique integer number such that
$$fh^{r}f^{-1}(p_{f (I)}) \leq f(p_I) < fh^{r+1}f^{-1}(p_{f(I)}).$$

\vspace{0.2cm}

\begin{lem}
For all $f,g$ in $G$ one has
$$\ell_I (g \circ f) \leq \ell_{f(I)}(g) + \ell_I (f) + 1.$$
\label{lema1}
\end{lem}

\begin{proof} Let $r$ be the unique integer number such that
\begin{equation}
(fhf^{-1})^{r}(p_{f(I)}) \leq f(p_I) < (fhf^{-1})^{r+1}(p_{f(I)}),
\label{defn-r}
\end{equation}
and let $s$ be the unique integer number such that
$$(gfhf^{-1}g^{-1})^{s}(p_{gf(I)}) \leq g(p_{f(I)}) < (gfhf^{-1}g^{-1})^{s+1}(p_{gf(I)}),$$
so that
$$l_I(f)=\lvert r \lvert, \quad l_{f(I)}(g)= \lvert s \lvert.$$
We then have
$$g^{-1}(gfhf^{-1}g^{-1})^{s}(p_{gf(I)}) \leq p_{f(I)} < g^{-1}(gfhf^{-1}g^{-1})^{s+1}(p_{gf(I)}),$$
that is
$$(fhf^{-1})^{s}g^{-1}(p_{gf(I)}) \leq p_{f(I)} < (fhf^{-1})^{s+1}g^{-1}(p_{gf(I)}).$$
Therefore,
$$(fhf^{-1})^{r}(fhf^{-1})^{s}g^{-1}(p_{gf(I)}) \leq f(p_I)
< (fhf^{-1})^{r+1}(fhf^{-1})^{s+1}g^{-1}(p_{gf(I)}),$$
and hence
$$(fhf^{-1})^{r+s}g^{-1}(p_{gf(I)}) \leq f(p_I) < (fhf^{-1})^{r+s+2}g^{-1}(p_{gf(I)}).$$
This easily gives
$$g(fhf^{-1})^{r+s}g^{-1}(p_{gf(I)}) \leq gf(p_I) < g(fhf^{-1})^{r+s+2}g^{-1}(p_{gf(I)}),$$
and thus
$$(gfhf^{-1}g^{-1})^{r+s}(p_{gf(I)}) \leq gf(p_I) < (gfhf^{-1}g^{-1})^{r+s+2}(p_{gf(I)}).$$
This shows that $l_I(gf)$ equals either $\lvert r+s \lvert$ or
$\lvert r+s+1 \lvert$, which concludes the proof. \end{proof}

\vspace{0.35cm}

The following corollary is a direct consequence of 
the preceding lemma, but may be proved independently.

\vspace{0.1cm}

\begin{coro}
For every $f \!\in\! G$ one has
$$|\ell_I (f) - \ell_{f(I)} (f^{-1})| \leq 1.$$
\label{corolario}
\end{coro}

\begin{proof} From (\ref{defn-r}) one obtains
$$h^{-(r+1)}(p_{I}) < f^{-1}(p_{f(I)}) \leq h^{-r}(p_{I}) < h^{-r+1}(p_I),$$
and hence $\ell_{f(I)} (f^{-1})$ equals either $|r|$ or 
$|r+1|$. Since $\ell_I (f) = |r|$, the corollary follows.
\end{proof}

\vspace{0.2cm}


\section{The proof in the smooth case}

To rule out the possibility of ``concentration'' of the entropy on a type 2 connected
component $I$ of $\Omega^c$, in the $C^2$ case we will use classical control of distortion
arguments in order to construct, starting from the function $\ell_I$, a kind of
quasi-morphism from $G$ into $\mathbb{N}_0$. Slightly more generally, let $\mathcal{F}$
be any finite family of connected components of type 2 of $\Omega^c$. We denote by
$\mathcal{F}^{G}$ the family formed by all the intervals contained in the orbits
of the intervals in $\mathcal{F}$. For each $f \in G$ we then define
$$\ell_{\mathcal{F}} (f) = \sup_{I \in \mathcal{F}^{G}} \ell_{I} (f).$$
{\em A priori}, the value of $\ell_{\mathcal{F}}$ could be infinite.
We claim however that, for groups of $C^2$ diffeomorphisms, its
value is necessarily finite for every element $f$.

\vspace{0.1cm}

\begin{prop} For every finite family $\mathcal{F}$ of type 2 connected components
of $\Omega^c$, the value of $\ell_{\mathcal{F}}(f)$ is finite for each $f \in G$.
\label{finito}
\end{prop}

\vspace{0.1cm}

To show this proposition, we will need to estimate the function 
$\ell_I (f)$ in terms of the distortion of $f$ on the interval $I$.

\vspace{0.15cm}

\begin{lem}
For each fixed type 2 connected component $I$ of $\Omega^c$ and every $g \!\in\! G$,
the value of $\ell_I (g)$ is bounded from above
by a number $L (V)$ depending on $V \!=\! var(\log(g'|_I))$,
the total variation of the logarithm of the derivative of the restriction of $g$ to $I$.
\label{lema-n}
\end{lem}

\begin{proof} Denote $]a,b[ = \! I$ and $]\bar{a},\bar{b}[ = \! g(I)$.
If $h$ is a generator for the stabilizer
of $I$, then for every $f \!\in\! G$ the value of $\ell_{I}(f)$ corresponds (up to some constant
$\pm 1$) to the number of fundamental domains for the dynamics of $fhf^{-1}$ on $f(I)$ between
the points $p_{f(I)}$ and $f(p_I)$, which in its turn corresponds to the number of fundamental
domains for the dynamics of $h$ on $I$ between $f^{-1}(p_{f(I)})$ and $p_I$. Therefore, we need
to show that there exists $c < d$ in $]a,b[$ depending on $V$ and such that $g^{-1} (p_{g(I)})$
belongs to $[c,d]$. We will show that this happens for the values
$$c = a + \frac{|I|}{2 e^V} \qquad \mbox{ and } \qquad d = b - \frac{|I|}{2 e^V}.$$
We will just check that the first choice works, leaving the second one to the reader. By the
Mean Value Theorem, there exists $x \!\in\! g(I)$ and $y \in [\bar{a},p_{g(I)}]$ such that
$$(g^{-1})'(x) = \frac{|I|}{|g(I)|}$$
and
$$(g^{-1})'(y) = \frac{|g^{-1}([\bar{a},p_{f(I)}])|}{|[\bar{a},p_{g(I)}]|}
= \frac{g^{-1}(p_{g(I)}) - a}{|g(I)| /2}.$$
By the definition of the constant $V$, we have
$(g^{-1})'(x) / (g^{-1})'(y) \leq e^V$. This gives
$$e^V \geq \frac{|I| / |g(I)|}{2 (g^{-1}(p_{g(I)}) - a) / |g(I)|}
= \frac{|I|}{2 (g^{-1}(p_{g(I)}) - a)},$$
thus proving that $g^{-1}(p_{g(I)}) \geq a + \frac{|I|}{2 e^V}$,
as we wanted to show.
\end{proof}

\vspace{0.15cm}

\noindent{\em Proof of Proposition \ref{finito}.} Let $J \! = ]\bar{a},\bar{b}[$ be
an interval in the orbit by $G$ of $I \!= ]a,b[$. If $g = g_{i_n} \cdots g_{i_1}$, 
$g_{i_j} \!\in\! \Gamma$, is an element of minimal length sending $I$ into $J$, 
then the intervals 
$I, g_{i_1}(I), g_{i_2}g_{i_1}(I),\ldots,g_{i_{n-1}} \cdots g_{i_2} g_{i_1} (I)$
have two by two disjoint interiors. Therefore,
$$var(\log (g'|_I) )
\leq \sum_{j=0}^{n-1} var(\log (g_{i_{j+1}}' |_{g_{i_j} \cdots g_{i_1} (I)} )
\leq \sum_{h \in \Gamma} var (\log (h')) =: W.$$
Moreover, denoting $V \!=\! var(\log(f'))$,
$$var (\log( (fg)'|_I )) \leq var(\log(g'|_I)) + var (\log(f')) = W + V.$$
By Lemmas \ref{lema1} and \ref{lema-n} and Corollary \ref{corolario},
\begin{eqnarray*}
\ell_J (f)
&\leq& \ell_J (g^{-1}) + \ell_I (fg) + 1 \\
&\leq& \ell_I (g) + \ell_I (fg) + 2 \\
&\leq& L(W) + L( W + V ) +2.
\end{eqnarray*}
This shows the proposition when $\mathcal{F}$ consists of a single interval.
The case of general finite $\mathcal{F}$ follows easily. $\hfill\square$

\vspace{0.5cm}

For a given $\varepsilon \!>\! 0$ we define $\ell_{\varepsilon} = \ell_{\mathcal{F}_{\varepsilon}}$,
where $\mathcal{F}_{\varepsilon} \!=\! \{ I_1, \ldots, I_k \}$ is the family of the connected
components of $\Omega^c$ having length greater than or equal to $\varepsilon$, with $k \!=\!
k(\varepsilon)$. Notice that, by Lemma \ref{lema1}, for every $f,g$ in $\Gamma$ one has
\begin{equation}
\ell_{\varepsilon} (gf) \leq \ell_{\varepsilon}(g) + \ell_{\varepsilon}(f) + 1
\label{casi-sub}
\end{equation}

\vspace{0.1cm}

\begin{lem} There exists constants $A(\varepsilon) \! > \! 0$ and $B(\varepsilon)$ satisfying
the following property: If $x_1,\ldots,x_m$ are points contained in a single connected
component of type 2 of $\Omega^c$ and \,$x_i,x_j$\, are $(\varepsilon,n)$-separated
for every $i \neq j$, then $m \leq A(\varepsilon) n + B(\varepsilon)$.
\label{lema-5}
\end{lem}

\begin{proof} Denote $c_{\varepsilon} \!\!= \max \{\ell_{\varepsilon}(g)\!\!: g \in \Gamma \}$
(according to Proposition \ref{finito}, the value of $c_{\varepsilon}$ is finite). Let $I$ be the
type 2 connected component of $\Omega^c$ containing $x_1,\ldots,x_m$. We may assume that
$x_1 < x_2 < \ldots < x_m$. For each $1 \leq i \leq k$ let $h_i$ be the generator of
$Est(I_i)$. Notice that $\ell_{\varepsilon} (h_i^{r}) \geq \lvert r \lvert$ for all $r \in \mathbb{Z}$.

If $f$ is an element in $B_{\Gamma}(n)$ sending $I$ into some $I_i$, then the number of points
which are $\varepsilon$-separated by $f$ is less than or equal to $1 / \varepsilon + 1$.
We claim that the number of elements in $B_{\Gamma}(n)$ sending $I$ into $I_i$ is
bounded above by $4n c_{\varepsilon} + 4n - 1$. Indeed, if $g$ also sends $I$
onto $I_i$ then $g f^{-1} \!\in\! Est(I_i)$, hence $g f^{-1} \!=\! h_i^{r}$
some r. Therefore, using (\ref{casi-sub}) one obtains
$\lvert r \lvert \leq \ell_{\varepsilon} (h_i^{r}) \leq 2 n c_{\varepsilon} + 2n -1$.

Since the previous arguments apply to each type 2 interval $I_i$, we have
$$m \leq k \big( \frac{1}{\varepsilon} + 1 \big) (4nc_{\varepsilon} + 4n - 1).$$
Therefore, letting
$$A(\varepsilon) = \big( 4k + \frac{4k}{\varepsilon} \big) (1 + c_{\varepsilon})
\qquad \mbox{and} \qquad B(\varepsilon) = - \big( k + \frac{k}{\varepsilon} \big),$$
this concludes the proof.
\end{proof}

\vspace{0.5cm}

To conclude the proof of Theorem A, the following notation will be useful.

\vspace{0.1cm}

\begin{nott} Given $\varepsilon > 0$ and $n \in \mathbb{N}$, we will denote by
$\,s(n,\varepsilon)$\, the largest cardinality of a $\,(n,\varepsilon)-$separated
subset of $S^1$. Likewise, $s_{\Omega}(n,\varepsilon)$ will denote the largest
cardinality of a $\,(n,\varepsilon)-$separated set contained in the
non-wandering set.
\end{nott}

\vspace{0.1cm}

\noindent{\em Proof of Theorem A.} Fix $0 \!<\! \varepsilon \!<\! 1/2L$, where $L$ is a
common Lipschitz constant for the elements in $\Gamma$. We will show that, for some
function $p_{\varepsilon}$ growthing linearly on $n$ (and whose coefficients
depend on $\varepsilon$), one has
\begin{equation}
s(n,\varepsilon) \leq p_{\varepsilon} (n) s_{\Omega} (n,\varepsilon) + p_{\varepsilon}(n).
\label{fundamental}
\end{equation}
Actually, any function $p_{\varepsilon}$ with sub-exponential growth and verifying such
an inequality suffices. Indeed, taking the logarithm in both sides, dividing by $n$,
and passing to the limit, this implies that
$$h_{\Gamma} (G \ac \mathrm{S}^1,\varepsilon) = h_{\Gamma} (G \ac \Omega,\varepsilon).$$
Letting $\varepsilon$ go to zero, this gives
$$h_{\Gamma} (G \ac \mathrm{S}^1) \leq h_{\Gamma} (G \ac \Omega).$$
Since the opposite inequality is obvious, this shows the desired
equality between the entropies.

To show (\ref{fundamental}), fix a $(n,\varepsilon)$-separated set $S$ containing $\,s(n,\varepsilon)$
points. Let $n_\Omega$ (resp. $n_{\Omega^{c}}$) be the number of points in $S$ which are in
$\Omega$ (resp. in $\Omega^c$). Obviously, $\,s(n,\varepsilon) = n_\Omega + n_{\Omega^c}$.
Let $t \!=\! t_S$ be the number of connected components of $\Omega^{c}$ containing points in $S$,
and let $l = [\frac{t}{2}]$, where $[\cdot]$ denotes the integer part function. We will show
that there exists a $(n,\varepsilon)$-separated set $T$ contained in $\Omega$ having
cardinality $l$. This will obviously give \,$s_{\Omega} (n,\varepsilon) \geq l$.\, Using
the inequalities \,$t \leq 2l+1$\, and \,$n_{\Omega} \leq s_{\Omega}(n,\varepsilon)$,\,
and by Lemmas \ref{lema0} and \ref{lema1}, this will imply that
\begin{eqnarray*}
s(n,\varepsilon)
&=& n_{\Omega} + n_{\Omega^c} \\
&\leq& n_{\Omega} + t k 
   \big( 1 + \frac{1}{\varepsilon} \big) (4 n c_{\varepsilon} + 4n -1) \\
&\leq& s_{\Omega} (n,\varepsilon) + (2 s_{\Omega}(n,\varepsilon) + 1)
   k \big( 1 + \frac{1}{\varepsilon} \big) (4 n c_{\varepsilon} + 4n -1),
\end{eqnarray*}
thus showing (\ref{fundamental}).

To show the existence of the set $T$ with the properties above, we proceed in a
constructive way. Let us number the connected components of $\Omega^c$ containing
points in $S$ in a cyclic way by $I_1,\ldots,I_t$. Now for each $1 \leq i \!\leq\! l$
choose a point $t_i \!\in\! \Omega$ between $I_{2i-1}$ and $I_{2i}$, and let
\,$T \!=\! \{t_1,\ldots,t_l\}$.\, We need to check that, for \,$i \!\neq\! j$,\,
the points $t_i$ and $t_j$ are $(n,\varepsilon)$-separated. Now by construction,
for each $i \!\neq\! j$ there exist at least
two different points $x,y$ in $S$ contained
in the interval of smallest length in $\mathrm{S}^1$ joining $t_i$ and $t_j$.
Since $S$ is a $(n,\varepsilon)$-separated set, there exist $m \!\leq\! n$
and $g_{i_1},\ldots,g_{i_m}$ in $\Gamma$ so that
$dist \big( h(x), h(y) \big) \geq \varepsilon$, where
$h \!=\! g_{i_m} \cdots g_{i_2} g_{i_1}$. Unfortunately, because of the topology
of the circle, this does not imply that
\,$dist \big( h(t_i), h(t_j) \big) \geq \varepsilon$.
However, the proof will be finished if we show that
\begin{equation}
dist \big( g_{i_r} \cdots g_{i_1} (t_i), g_{i_r} \cdots g_{i_1} (t_j) \big) \geq \varepsilon
\quad \mbox{for some} \quad 0 \leq r \leq m.
\label{ultima}
\end{equation}
This claim is obvious if \,$dist \big( t_i, t_j \big) \geq \varepsilon$.\, If this
is not the case then, by the definition of the constants $\varepsilon$ and $L$,
the length of the interval $[g_{i_1} (t_i),g_{i_1} (t_j)]$ is smaller than
$1/2$, and hence it coincides with the distance between its endpoints. If
this distance is at least $\varepsilon$, then we are done. If not, the
same argument shows that the length of the interval
$[g_{i_2} g_{i_1} (t_i), g_{i_2} g_{i_1} (t_j)]$ is smaller than
$1/2$ and coincides with the distance between its endpoints. If
this length is at least $\varepsilon$, then we are done. If not, we continue
the procedure... Clearly, there must be some integer $r \!\leq\! m$ such that
the length of the interval
$[g_{i_{r-1}} \cdots g_{i_1} (t_i), g_{i_{r-1}} \cdots g_{i_1} (t_j)]$ is smaller
than $\varepsilon$, but the one of
$[g_{i_{r}} \cdots g_{i_1} (t_i), g_{i_{r}} \cdots g_{i_1} (t_j)]$ is greater than or
equal to $\varepsilon$. As before, the length of the later interval will be forced to
be smaller than $1/2$, and hence it will coincide with the distance between its endpoints.
This shows (\ref{ultima}) and concludes the proof of Theorem A. $\hfill\square$

\vspace{0.2cm}


\section{The proof in the case of non existence of sub-exponentially distorted elements}

Recall that the topological entropy is invariant under topological conjugacy. Therefore, due
to \cite[Th\'eor\`eme D]{DKN}, in order to prove Theorem B we may assume that $G$ is a group
of bi-Lipschitz homeomorphisms. Let $L$ be a common Lipschitz constant for the elements in
$\Gamma$. Fix again $0 \!<\! \varepsilon \!<\! 1/2L$, and let $I_1,\ldots,I_{k}$
be the connected components of $\Omega^c$ having length greater than or equal to $\varepsilon$.
Let $h_i$ be a generator for the stabilizer of $I_i$ (with $h_i \!=\! Id$ in case where $I_i$
is of type 1). Consider the minimal non decreasing function $q_{\varepsilon}$ such that, for
each of the nontrivial $h_i$'s, one has \,$q_{\varepsilon} (\|h_i^r\|) \geq r$ for all 
positive $r$. We will
show that (\ref{fundamental}) holds for the function
$$p_{\varepsilon} (n) = 2 k 
\big(1 + \frac{1}{\varepsilon} \big) (2 q_{\varepsilon} (2n) + 1) + 1 .$$
Notice that, by assumption, this function $p_{\varepsilon}$ growths at most sub-exponentially
on $n$. Hence, as in the case of Theorem A, inequality (\ref{fundamental}) allows to finish
the proof of the equality between the entropies.

The main difficulty for showing (\ref{fundamental}) in this case is that
Lemma \ref{lema-5} is no longer available. However, the following still holds.

\vspace{0.2cm}

\begin{lem} If $x_1,\ldots,x_m$ are points contained in a single type 2 connected
component $I$ of $\Omega^c$ having length at least $\varepsilon$, and \,$x_i,x_j$\,
are $(\varepsilon,n)$-separated for every $i \neq j$, then $m \leq
k \big( \frac{1}{\varepsilon} + 1 \big) (2 q_{\varepsilon} (2n) + 1)$.
\label{lema2}
\end{lem}

\begin{proof} Let $I$ be the type 2 connected component of $\Omega^c$ containing
$x_1,\ldots,x_m$. We may assume that $x_1 < x_2 < \ldots < x_m$. If $f$ is an
element in $B_{\Gamma}(n)$ sending $I$ into some $I_i$, then the number of points
which are $\varepsilon$-separated by $f$ is less than or equal to $1 / \varepsilon + 1$.
We claim that the number of elements in $B_{\Gamma}(n)$ sending $I$ into $I_i$ is
bounded above by $q_{\varepsilon}(r)$. Indeed, if $g$ also sends $I$
onto $I_i$ then $g f^{-1} \!\in\! Est(I_i)$, hence $g f^{-1} \!=\! h_i^{r}$
some r. Therefore,
$$2 n \geq \| gf^{-1} \| = \| h_i^r \|,$$
and hence
$$q_{\varepsilon}(2n) \geq q_{\varepsilon} (\| h_i^r \|) \geq |r|.$$
Since the previous arguments apply to each type 2 interval $I_i$, this gives
$$m \leq k \big( \frac{1}{\varepsilon} + 1 \big) (2 q_{\varepsilon} (2n) + 1),$$
thus proving the lemma.
\end{proof}

To show (\ref{fundamental}) in the present case, we proceed as in the proof of Theorem A. We fix
a $(n,\varepsilon)$-separated set $S$ containing $\,s(n,\varepsilon)$ points. We let $n_\Omega$
(resp. $n_{\Omega^{c}}$) be the number of points in $S$ which are in $\Omega$ (resp. in
$\Omega^c$), so that $\,s(n,\varepsilon) = n_\Omega + n_{\Omega^c}$. Let $t \!=\! t_S$ be
the number of connected components of $\Omega^{c}$ containing points in $S$, and let
$l = [\frac{t}{2}]$, where $[\cdot]$ denotes the integer part function. As before, one can
show that there exists a $(n,\varepsilon)$-separated set $T$ contained in $\Omega$ having
cardinality $l$. This will obviously give \,$s_{\Omega} (n,\varepsilon) \geq l$.\,
Inequalities \,$t \leq 2l+1$\, and \,$n_{\Omega} \leq s_{\Omega}(n,\varepsilon)$\,
still holds. Using Lemmas \ref{lema0} and \ref{lema2} one now obtains
\begin{eqnarray*}
s(n,\varepsilon)
&=& n_{\Omega} + n_{\Omega^c} \\
&\leq& n_{\Omega} + t k 
   \big( 1 + \frac{1}{\varepsilon} \big) (2 q_{\varepsilon} (2n) + 1) \\
&\leq& s_{\Omega} (n,\varepsilon) + (2 s_{\Omega}(n,\varepsilon) + 1)
   k \big( 1 + \frac{1}{\varepsilon} \big) (2 q_{\varepsilon} (2n) + 1).
\end{eqnarray*}
This concludes the proof of Theorem B.

\vspace{0.5cm}


\noindent{\bf Acknowledgments.} I would like to thank Andr\'es Navas for introducing me to
this subject and his continuous support during this work, which was partially funded by
Research Network on Low Dimensional Dynamical Systems (PBCT-Conicyt's project ADI 17). 
I would also extend my gratitude to both the referee and the managing editor for 
pointing me a subtle error in the original version of this paper.


\vspace{0.25cm}

\begin{footnotesize}

\vspace{0.2cm}

\noindent Eduardo Jorquera\\
\noindent Dpto de Matem\'aticas, Fac. de Ciencias, Univ. de Chile\\
\noindent Las Palmeras 3425, $\tilde{\mathrm{N}}$u$\tilde{\mathrm{n}}$oa, Santiago, Chile\\
\noindent ejorquer@u.uchile.cl\\

\end{footnotesize}

\end{document}